\documentclass[11pt]{article}

\date{April 23, 2026}

\PassOptionsToPackage{dvipsnames,svgnames,x11names}{xcolor}
\usepackage{xcolor}
\definecolor{labelkey}{rgb}{0,0.08,0.45}
\definecolor{refkey}{rgb}{0,0.6,0.0}
\definecolor{Brown}{rgb}{0.45,0.0,0.05}
\definecolor{lime}{rgb}{0.00,0.8,0.0}
\definecolor{lblue}{rgb}{0.5,0.5,0.99}
\definecolor{OliveGreen}{rgb}{0,0.6,0}
\definecolor{tyrianpurple}{rgb}{0.4, 0.01, 0.24}
\definecolor{myseagreen}{HTML}{3FBC9D}
\definecolor{myblue}{rgb}{0.9,0.9,0.98}
\colorlet{hlcyan}{cyan!30}

\usepackage{mathpazo} % With old-style figures and real smallcaps
\usepackage[T1]{fontenc}
\usepackage{mathtools}
\usepackage{amssymb}
\usepackage{stmaryrd}
\usepackage{empheq}

\usepackage{hyperref}
\hypersetup{
  colorlinks=true,
  linkcolor=refkey,
  urlcolor=lblue,
  citecolor=magenta
}

\usepackage{amsthm}
\makeatletter
\def\th@plain{%
  \thm@notefont{}%
  \itshape % body font
}
\def\th@definition{%
  \thm@notefont{}%
  \normalfont % body font
}

\usepackage[capitalize,nameinlink]{cleveref}
\makeatother
\newtheorem{theorem}{Theorem}[section]

\newtheorem{corollary}[theorem]{Corollary}
\newtheorem{proposition}[theorem]{Proposition}

\newtheorem{example}[theorem]{Example}

\newtheorem{algorithm}[theorem]{Algorithm}
\newtheorem{remark}[theorem]{Remark}
\crefname{theorem}{Theorem}{Theorems}
\Crefname{theorem}{Theorem}{Theorems}
\crefname{fact}{Fact}{facts}
\Crefname{fact}{Fact}{facts}
\Crefname{algorithm}{Algorithm}{algorithms}
\crefname{algorithm}{Algorithm}{algorithms}
\crefname{equation}{}{equations}
\crefname{chapter}{Appendix}{chapters}
\crefname{item}{}{items}
\crefname{enumi}{}{}

\usepackage{nicematrix}
\usepackage{booktabs}
\usepackage{siunitx}
\usepackage{tikz}
\usepackage[color=myseagreen]{todonotes}
\usepackage{soul}
\usepackage{nameref}
\usepackage{graphicx}
\usepackage{float}
\usepackage[shortlabels,inline]{enumitem}
\setlist[enumerate]{nosep}
\usepackage{relsize}

\usepackage[margin=1in,footskip=0.25in]{geometry}
\makeatletter

\let\orig@label\label

\renewcommand{\label}[1]{%
  \begingroup
  \def\@currentlabelname{}%
  \ifx\current@theorem\relax\else
    \def\@currentlabelname{\current@theorem}%
  \fi
  \ifx\cref@currentlabel\undefined\else
    \let\@currentlabelname\cref@currentlabel
  \fi
  \orig@label{#1}%
  \endgroup
}

\AtBeginEnvironment{theorem}{\def\current@theorem{theorem}}
\AtBeginEnvironment{proposition}{\def\current@theorem{proposition}}
\AtBeginEnvironment{fact}{\def\current@theorem{fact}}
\makeatother

\newcommand{\seppthree}{\setlength{\itemsep}{-3pt}}

\newcommand{\nnn}{\ensuremath{{n\in\mathbb{N}}}}
\newcommand{\kkk}{\ensuremath{{k\in\mathbb{N}}}}
\newcommand{\thalb}{\ensuremath{\tfrac{1}{2}}}
\newcommand{\menge}[2]{\big\{{#1}~\big|~{#2}\big\}}

\newcommand{\scal}[2]{\left\langle{#1},{#2}\right\rangle}

\newcommand{\exi}{\ensuremath{\exists\,}}

\newcommand{\RR}{\ensuremath{\mathbb{R}}}

\newcommand{\NN}{\ensuremath{\mathbb{N}}}

\newcommand{\ran}{\ensuremath{\operatorname{ran}\,}}

\newcommand{\pinf}{\ensuremath{+\infty}}

\providecommand{\norm}[1]{\lVert#1\rVert}

\author{
  Heinz H.\ Bauschke\thanks{
    Mathematics, University of British Columbia,
    Kelowna, B.C.\ V1V~1V7, Canada. E-mail: \texttt{heinz.bauschke@ubc.ca}.}
  ~~~and~~~
  Tran Thanh Tung\thanks{
    Mathematics, University of British Columbia,
    Kelowna, B.C.\ V1V~1V7, Canada. E-mail: \texttt{tung.tran@ubc.ca}.}
}

\title{\textsf{Finite Termination of a Generalized Perceptron Algorithm}}

\begin{document}

\maketitle

\begin{abstract}

Motivated by Ridgway’s proof of the perceptron algorithm, 
we study a simple subgradient method for convex inequality 
systems in Hilbert space.
Assuming strict feasibility and bounded subgradients, 
we establish finite termination for several natural step sizes. 
We also examine what can go wrong without strict feasibility: finite convergence may fail even for one function, 
and with several functions the method may converge to a point outside the feasible set. 
The linear setting recovers the classical perceptron algorithm.
\end{abstract}



\section{Introduction}

Throughout this paper,
\begin{subequations}
\label{e:standing}
\begin{align}
  &\text{$X$ is a real Hilbert space, with inner product $\scal{\cdot}{\cdot}$ and induced norm $\norm{\cdot}$.}
\end{align}
We also assume that 
\begin{align}
&\text{\textbf{(nonempty finite index set)}\quad $I = \{1,2,\ldots,m\}$;}\\  
&\text{\textbf{(full domain)}\quad $(\forall i\in I)$ $f_i\colon X\to\RR$ is convex and continuous;}\\
&\text{\textbf{(nonempty feasible set)}\quad $C := \menge{x\in X}{\text{$f_i(x)\leq 0$ for all $i\in I$}}\neq\varnothing$;}\\
&\text{\textbf{(bounded subgradients)}\quad
$(\exists L> 0)$ 
$\displaystyle \max_{i\in I}\sup\|\partial f_i(X)\| \leq L <\pinf$.} 
\label{e:bdsub}
\end{align}
\end{subequations}

We will usually also assume strict feasibility: 
\begin{align}
\label{e:Slater}
&\text{
\textbf{(Slater point)} \quad $(\exi s\in X)$ 
$\displaystyle \max_{i\in I} f_i(s) < 0$; hence, 
$\displaystyle \sigma := \min_{i\in I} -f_i(s) > 0$.
}
\end{align}

Motivated by Ridgway's proof of the perceptron algorithm 
presented in \cite{BlockLevin}, we consider the following algorithm, 
for finding a feasible point in the convex set $C$. 

\begin{algorithm}
\label{algo}
Given a sequence of step sizes $(\alpha_k)_{\kkk}$ that 
satisfies  
\begin{align}
\label{e:deltak}
(\forall\kkk) \quad 
\alpha_k \geq 0 \;\;\text{and}\;\;
\delta_k := \alpha_k(2\sigma - \alpha_k L^2)
\end{align}
and a starting point $x_0\in X$, the algorithm proceeds as follows: 
If $x_k \in C$, then we are done and we stop. 
Otherwise, we update\footnote{For background on convex analysis 
in Hilbert space, see, e.g., \cite{BC2017}.} 
\begin{equation}
\label{e:step}
x_{k+1} := x_k - \alpha_k g_k,\quad 
\text{where $g_k\in\partial f_{i_k}(x_k)$ and $i_k\in I$ is such that $f_{i_k}(x_k) > 0$.}
\end{equation}
\end{algorithm}

Our main result (see \cref{t:finite} below) 
asserts the \emph{finite} convergence\footnote{For different, 
but similar in spirit, finite-convergence results, see, e.g., 
\cite{BWWX}, \cite{CCP}, \cite{DePIus} and \cite{Fuku}.} of $(x_k)_\kkk$
to a point in $C$ provided that \cref{e:standing} and \cref{e:Slater} hold and each $\delta_k$ is nonnegative with $\sum_{k=0}^\infty\delta_k=+\infty$. 

The remainder of this paper is organized as follows. 
The proof of the main result is provided in \cref{s:2} where 
we also discuss \cref{algo} in the absence of a Slater point. 
The connection to perceptrons is presented in \cref{s:3}. 

\section{Main result} 

\label{s:2}

\begin{proposition}[one-step estimate]
\label{p:1step}
Suppose that \cref{e:standing} and \cref{e:Slater} hold, 
$x_k\notin C$, and $x_{k+1}$ is obtained from 
\cref{e:step}. 
Then 
\begin{equation}
\|x_{k+1}-s\|^2 \leq \|x_k-s\|^2 - \delta_k. 
\end{equation}
\end{proposition}
\begin{proof}
Because $g_k\in\partial f_{i_k}(x_k)$, 
the subgradient inequality 
$f_{i_k}(x_k) + \scal{s-x_k}{g_k} \leq f_{i_k}(s)$ 
re-arranges to 
\begin{equation}
\label{e:slsubg}
\scal{x_k-s}{g_k} \geq f_{i_k}(x_k) - f_{i_k}(s) > \sigma > 0.
\end{equation}
It follows that 
\begin{align*}
\|x_{k+1}-s\|^2
&= 
\|(x_k-\alpha_kg_k)-s\|^2
= \|(x_k-s)-\alpha_kg_k\|^2\\
&= \|x_k-s\|^2 -2\alpha_k\scal{x_k-s}{g_k} + \alpha_k^2\|g_k\|^2
\\ 
&\leq 
\|x_k-s\|^2 - 2\alpha_k\sigma + \alpha_k^2 L^2,
\end{align*}
which yields the desired estimate.
\end{proof}

\begin{proposition}
\label{p:notermi}
Suppose that \cref{e:standing} and \cref{e:Slater} hold, 
each $\delta_k\geq 0$, and 
\cref{algo}
does not terminate after finitely many steps. 
Then
\begin{equation}
\sum_{k=0}^\infty \delta_k < +\infty.
\end{equation}
\end{proposition}
\begin{proof}
By \cref{p:1step}, we have $(\forall\kkk)$ 
$0\leq \delta_k \leq \|x_k-s\|^2 - \|x_{k+1}-s\|^2$. 
Hence $(\forall\nnn)$ $\sum_{k=0}^n \delta_k \leq \|x_0-s\|^2 - \|x_{n+1}-s\|^2 
\leq \|x_0-s\|^2$. 
\end{proof}

We are ready for our main result: 

\begin{theorem}[finite termination of \cref{algo}]
\label{t:finite}
Suppose that \cref{e:standing} and \cref{e:Slater} hold, 
each $\delta_k\geq 0$,
\begin{equation}
 \sum_{k=0}^\infty \delta_k = +\infty. 
\end{equation}
Then \cref{algo}
finds a point in $C$ after finitely many steps\footnote{A second inspection of the proof reveals that \cref{e:bdsub} can be replaced by 
$(\exists L>0)$ $\sup_\kkk\|g_k\|\leq L<+\infty$.}. 
\end{theorem}
\begin{proof}
This is the contrapositive of \cref{p:notermi}.
\end{proof}

\begin{corollary}[finite termination with constant step size]
\label{c:constfin}
Suppose that \cref{e:standing} and \cref{e:Slater} hold, 
and 
\begin{equation}
\alpha_k \equiv \alpha, \quad 
\text{where } 
0<\alpha<\frac{2\sigma}{L^2}.
\end{equation}
Then \cref{algo}
finds a point in $C$ after finitely many steps. 
\end{corollary}
\begin{proof}
  Note that 
$\delta_k \equiv \alpha(2\sigma - \alpha L^2) > 0$. 
Now apply \cref{t:finite}.
\end{proof}

In practice, 
one might know that $\sigma$ and $L$ exist, but one might not know 
their actual values --- in that case, the following result is useful:

\begin{corollary}
\label{c:classic}
Suppose that \cref{e:standing} and \cref{e:Slater} hold, 
and the nonnegative sequence $(\alpha_k)_\kkk$ of step sizes satisfies 
\begin{equation}
\label{e:classic}
\sum_{k=0}^\infty \alpha_k=+\infty 
\quad\text{and}\quad 
\sum_{k=0}^\infty \alpha_k^2<+\infty.
\end{equation}
(For instance, we could pick $(\alpha_k)_\kkk = (1/(k+1))_\kkk$.)
Then \cref{algo}
finds a point in $C$ after finitely many steps. 
\end{corollary}
\begin{proof}
The assumptions imply that there exists $k_0\in\NN$ such that 
$(\forall k\geq k_0)$ $\delta_k\geq 0$ and 
$\sum_{k\geq k_0}\delta_k=+\infty$. 
The result thus follows from \cref{t:finite} (with starting point $x_{k_0}$). 
\end{proof}

If $m=1$, i.e., $I=\{1\}$, and there is no Slater point, then one can 
at least guarantee weak (but not necessarily finite) 
convergence in \cref{c:classic}:

\begin{remark}[no Slater point and only one function yields 
weak convergence]
\label{r:classic}
Suppose that \cref{e:standing} holds with $m=1$, but 
$\min f_1(X)=0$ (so Slater's condition \cref{e:Slater} fails). 
Furthermore, assume that the sequence of nonnegative step sizes $(\alpha_k)_\kkk$ satisfies
\cref{e:classic}. 
Then \cref{algo}
produces a sequence 
$(x_k)_\kkk$ that converges weakly to a point in $C$. 
Indeed, set $(\forall \kkk)$
$\eta_k := \max\{1,\|g_k\|\}$ and $\widetilde{\alpha_k} := \alpha_k{\color{red}/}\eta_k$. 
Then the result follows from 
Alber, Iusem, and Solodov's \cite[Theorem~1]{AIS} 
(applied with $(\widetilde{\alpha}_k)_\kkk$). 

However, \emph{finite convergence is no longer guaranteed}: 
Suppose that $X=\RR$ and that $f_1=H$ is the 
\emph{Huber function}
$H(x)=\thalb x^2$ if $|x|\leq 1$; 
$H(x)= |x|-\thalb$ if $1\leq|x|$. 
Here $C=\{0\}$ but 
no Slater point exists! Consider 
$(\alpha_{k})_\kkk = (1/(k+2))_\kkk$ and $x_0=1$. 
Then $(x_k)_\kkk = (1/(k+1))_\kkk$ converges to $0\in C$; 
however, the convergence is not finite.
\end{remark}

In contrast, if $m\geq 2$ and there is no Slater point, then 
\cref{algo} in the context of \cref{c:classic} fails spectacularly: 

\begin{example}[no Slater and two functions may fail to yield a solution]
\label{ex:noSl2}
Suppose that $X=\RR$, $m=2$, and 
\begin{equation}
f_1\colon X\to\RR\colon x\mapsto 
\begin{cases}
0, &\text{if $x\leq 0$;}\\
\thalb x^2, &\text{if $0\leq x\leq 1$;}\\
x-\thalb, &\text{if $1\leq x$}
\end{cases}
\end{equation}
is a \emph{truncated} Huber function. 
As in \cref{r:classic}, we pick 
$(\alpha_{k})_\kkk = (1/(k+2))_\kkk$ and $x_0=1$ so that 
$(x_k)_\kkk = (1/(k+1))_\kkk$ converges to $0\in C$. 
Note that $\ran\partial f_1 = [0,1]$, 
so we can and do assume that $L=1$. 
Now suppose that\footnote{We have considerable flexibility in picking $f_2$: indeed, any convex continuous function $f_2$ with 
$f_2(0)>0$ and for which 
$(\exists x<0)$ $f_2(x)\leq 0$ will do the trick!} $f_2(x)=x+1$. 
Then $(x_k)_\kkk$ is a valid incarnation of \cref{algo}. 
However, the limit of $(x_k)_\kkk$ satisfies $f_2(0)=1>0$ and 
so $0\notin C$.
\end{example}

\begin{remark} 
In \cref{ex:noSl2}, \cref{algo} does not work because
it picked the ``wrong'' subgradient for the violation encountered. 
If one sticks with the gradient from the most violated constraint, 
then that particular method converges weakly, again by applying 
\cite[Theorem~1]{AIS} to $\max_{i\in I}f_i$. 
\end{remark}

\section{Perceptron setting}

\label{s:3}

Now assume that 
\begin{equation}
(\forall i\in I) \quad 
f_i = \scal{\cdot}{-a_i}.
\end{equation}
Note that 
\begin{equation}
\nabla f_i(x) =-a_i, 
\end{equation}
and so the bounded subgradient assumption holds with 
\begin{equation}
L := \max_{i\in I} \|a_i\| < +\infty.
\end{equation}
Assume that $z\in X$ is a strict feasible solution, i.e., 
\begin{equation}
(\forall i\in I) \quad
\scal{z}{a_i} > 0.
\end{equation}
Set 
\begin{equation}
\mu := \min_{i\in I} \scal{z}{a_i} > 0.
\end{equation}
Then 
\begin{equation}
(\forall \rho>0) \quad
s_\rho := \rho\frac{L^2}{\mu} z
\end{equation}
is also a Slater point. Moreover, 
because 
$(\forall i\in I)$ 
$-f_i(s_\rho)= \scal{s_\rho}{a_i} = \rho\frac{L^2}{\mu}\scal{z}{a_i}$, 
we have 
\begin{equation}
\sigma_\rho := \min_{i\in I} -f_i(s_\rho) = \rho L^2.
\end{equation}
This gives us massive flexibility for picking the 
sequence of step sizes! 
For instance, if we wish to work with a constant step size, 
then by \cref{c:constfin} we 
can pick any step size $\alpha$ such that
\begin{equation}
\label{e:260423b}
0<\alpha < \frac{2\sigma_\rho}{L^2} = 2\rho.
\end{equation}
Because $\rho$ can be chosen arbitrarily large, 
it means that we can pick \emph{any} $\alpha>0$ in \cref{c:constfin}! 

The above flexibility, however, relies crucially on the Slater's condition. The following example shows that, when there is no Slater point, for any fixed $\alpha>0$, there exists an initial point for which the iterates fail to converge:
\begin{example}[no Slater and two functions may fail to yield convergence]
\label{ex:noSp}
Suppose that $X=\RR$, $m=2$, $a_1=1$, and $a_2=-1$. Then $C=\{0\}$, and Slater's condition \cref{e:Slater} fails. For a fixed stepsize $\alpha>0$, and assuming $x_k \neq 0$, the next iterate is given by
\begin{equation}\label{260423a}
x_{k+1} =
\begin{cases}
x_k - \alpha, & \text{if } x_k > 0;\\
x_k + \alpha, & \text{if } x_k < 0.
\end{cases}
\end{equation}
Now let $x_0 \in \left]0,\alpha\right[$. 
It follows from \cref{260423a} that, for all $k\in\mathbb{N}$,
\begin{equation}
\begin{cases}
x_{2k} = x_0,\\
x_{2k+1} = x_0 - \alpha.
\end{cases}
\end{equation}
Hence, the sequence $(x_k)_{k\in\mathbb{N}}$ is periodic and does not converge.
\end{example}

Returning to \cref{e:260423b}, we note that 
for $\rho=1$, we can pick $\alpha=1$ and 
hence recover the classical perceptron algorithm,
with W.C.\ Ridgway's proof as presented\footnote{Unfortunately, Ridgway's name is misspelled in \cite{BlockLevin}. See also \cite{Ridgway}.} in 
Block and Levin's \cite[Section~2]{BlockLevin}:
\begin{equation}
\label{percep}
x_{k+1} = x_k + a_{i_{k}}. 
\end{equation}
For a classical proof that the algorithm finds a point in $C$ after finitely many steps see \cite[Theorem~11.1]{MP}. 
If $C=\varnothing$, then the sequence generated by \cref{percep} stays 
bounded (see \cite[Theorem~1]{BlockLevin}). 
It would be interesting to find out what can be said for 
the much more general \cref{algo}.

\subsection*{Statements and Declarations}

\noindent\textbf{Acknowledgments.} We thank Dr.~Yair Censor for bringing \cite{BlockLevin} to our attention.

\noindent\textbf{Funding.} The research of HHB was partially supported by a Discovery Grant from the Natural Sciences and Engineering Research Council of Canada.

\noindent\textbf{Data availability.} 
No datasets were generated or analyzed for the research described in this article.

\noindent\textbf{Conflict of interest.} The authors declare that there is no conflict of interest in the publication of this paper.

\noindent\textbf{Ethical approval.} This article does not contain any studies with human participants or animals performed by any of the authors.

\end{document}